\documentclass[12pt]{amsart}%
\usepackage{amscd,amsmath,latexsym,amsthm,amsfonts,amssymb,graphicx,color,geometry,hyperref}
\usepackage[utf8]{inputenc}
\usepackage{amsmath}
\usepackage{amsfonts}
\usepackage{amssymb}
\usepackage{color, soul}
\usepackage{graphicx}
\usepackage{amsmath,amsthm,amsfonts,tikz}%
\providecommand{\U}[1]{\protect\rule{.1in}{.1in}}
\providecommand{\U}[1]{\protect\rule{.1in}{.1in}}

\newtheorem{theorem}{Theorem}[section]
\newtheorem{proposition}[theorem]{Proposition}

\newtheorem{remark}[theorem]{Remark}

\newtheorem{lemma}[theorem]{Lemma}

\newtheorem{definition}[theorem]{Definition}
\numberwithin{equation}{section}

\begin{document}
\title[On the Maurey--Pisier and Dvoretzky--Rogers theorems]{On the Maurey--Pisier and Dvoretzky--Rogers theorems}
\author[Ara\'ujo]{G. Ara\'ujo}
\address[G. Ara\'ujo]{Departamento de Matem\'{a}tica \\
	\indent Universidade Estadual da Para\'{\i}ba \\
	58.429-600 - Campina Grande \\
	Brazil}
\email{gdasaraujo@gmail.com}
\author[Santos]{J. Santos}
\address[J. Santos]{Departamento de Matem\'{a}tica \\
\indent
	Universidade Federal da Para\'{\i}ba \\
\indent
	58.051-900 - Jo\~{a}o Pessoa, Brazil}
\email{joedsonmat@gmail.com or joedson@mat.ufpb.br}
\thanks{J. Santos was supported by Conselho Nacional de Desenvolvimento Científico e Tecnológico--CNPq.}

\begin{abstract}
A famous theorem due to Maurey and Pisier asserts that for an infinite dimensional Banach space $E$, the infumum of the $q$ such that the identity map $id_{E}$ is absolutely $\left(  q,1\right)  $-summing is precisely $\cot E$. In the same direction, the Dvoretzky--Rogers Theorem asserts $id_{E}$ fails to be absolutely $\left(  p,p\right)  $-summing, for all $p\geq1$. In this note, among other results, we unify both theorems by charactering the parameters $q$ and $p$ for which the identity map is absolutely $\left(  q,p\right)$-summing. We also provide a result that we call \textit{strings of coincidences} that characterize a family of coincidences between classes of summing operators. We illustrate the usefulness of this result by extending classical result of Diestel, Jarchow and Tonge and the coincidence result of Kwapie\'{n}. 
\end{abstract}

\maketitle


\section{Introduction and background}

Let $2\leq q<\infty$. A Banach space $E$ has cotype $q$ (see \cite[page
218]{diestel}) if there is a constant $C>0$ such that, no matter how we select
finitely many vectors $x_{1},\dots,x_{n}\in E$,%
\begin{equation}
\left(  \sum_{k=1}^{n}\Vert x_{k}\Vert^{q}\right)  ^{\frac{1}{q}}\leq C\left(
\int_{[0,1]}\left\Vert \sum_{k=1}^{n}r_{k}(t)x_{k}\right\Vert ^{2}dt\right)
^{1/2}, \label{99}%
\end{equation}
where $r_{k}$ denotes the $k$-th Rademacher function. In this context we also define
\[
\cot E:=\inf\{q \ : \ E \ \text{has cotype} \ q\}.
\]


Recall that if $1\leq p\leq q\leq\infty$, for Banach spaces $E,F$, a linear
operator $u:E\rightarrow F$ is absolutely $\left(  q,p\right)  $-summing if $\left(  u(x_{j})\right)  _{j=1}^{\infty}\in\ell_{q}(F)$ whenever $\left(x_{j}\right)  _{j=1}^{\infty}\in\ell_{p}^{w}(E)$; we recall that $\ell_{p}^{w}(E)$ is the linear space of the sequences $\left(  x_{j}\right)_{j=1}^{\infty}$ in $E$ such that $\left(  \varphi\left(  x_{j}\right)\right)  _{j=1}^{\infty}\in\ell_{p}$ for every continuous linear functional $\varphi:E\rightarrow\mathbb{K}$. The class of all absolutely $(q,p)$-summing operators from $E$ to $F$ will be denoted by $\Pi_{(q,p)}(E;F)$. We denote $\Pi_{(p,p)}(E;F)$ by $\Pi_{p}(E;F)$. From now on, for any $p>1$ the symbol
$p^{\ast}$ denotes the conjugate of $p$, i.e., $p^{\ast}=p/\left(  p-1\right)
.$ Cotype and absolutely summing operators are closely related by the famous
Maurey--Pisier Theorem (see also \cite{GPR} for further relations between
cotype and absolutely summing operators):

\begin{theorem}[Maurey--Pisier]\label{pisier}
For every infinite dimensional Banach space $E$, we have $\cot E=\inf$\{$a:id_{E}$ is absolutely $\left(  a,1\right)  $-summing\}.
\end{theorem}

In the same direction, the Dvoretzky--Rogers Theorem \cite{DR} tells
us that $id_{E}$ is not absolutely $(p,p)$-summing, regardless of the
$p\geq1.$ In a more general version, as stated in \cite[Theorem 10.5]%
{diestel}, it reads as follows:

\begin{theorem}[Dvoretzky--Rogers]\label{dvo}
If $E$ is an infinite dimensional Banach space, the identity map $id_{E}$ is not absolutely $(q,p)$-summing whenever $\frac{1}{p}-\frac{1}{q}<\frac{1}{2}$.
\end{theorem}

Our first main result revisits and unifies both theorems. Our second main result provides a family of coincidences for the classes of absolutely summing operators that encompasses the following classical results:

\begin{theorem}[Diestel--Jarchow--Tonge]\label{djt}Let $E$ and $F$ be Banach spaces

(a) If $E$ has cotype $2$, then $\Pi_{2  }\left(  E;F\right)  =\Pi_{1}\left(  E;F\right)$.

(b) If $E$ has cotype $2<q<\infty$, then $\Pi_{r }\left(  E;F\right)  =\Pi_{1 }\left(  E;F\right)$ for all $1<r<q^*$.
\end{theorem}

\begin{theorem}[\cite{kwapien}, Kwapie\'{n}]\label{kw}
Let $1\leq p\leq\infty.$ Then
\[
\Pi_{\left(  r(p),1\right)  }\left(  \ell_1;\ell_p\right)  =\mathcal{L}\left(  \ell_1;\ell_p\right)  \ \text{where}\ \ \frac{1}{r(p)}=1- \left|\frac{1}{p}-\frac{1}{2}\right|.
\]
Moreover, if $r<r(p)$, then $\Pi_{\left(  r,1\right)  }\left(  \ell_1;\ell_p\right)  \neq\mathcal{L}\left(  \ell_1;\ell_p\right)$.
\end{theorem}

\section{Main results}


The next definition was introduced by M.C. Matos in \cite{Matos} and, independently, by D. P\'{e}rez-Garc\'{\i}a in \cite{perez}.

\begin{definition}
	\label{oms} Let $1\leq p\leq q\leq\infty$. A multilinear operator
	$T:E_{1}\times\cdots\times E_{m}\rightarrow F$ is multiple $(q,p)$-summing if
	there exist a constant $C>0$ such that
	\[
	\left(  \sum\limits_{j_{1},...,j_{m}=1}^{\infty}\left\Vert T(x_{j_{1}}%
	^{(1)},\dots,x_{j_{m}}^{(m)})\right\Vert ^{q}\right)  ^{\frac{1}{q}}\leq
	C\prod\limits_{k=1}^{m}\left\Vert (x_{j}^{(k)})_{j=1}^{\infty}\right\Vert
	_{w,p}%
	\]
	for all $(x_{j}^{(k)})_{j=1}^{\infty}\in\ell_{p}^{w}\left(  E_{k}\right)  $,
	with $k=1,...,m$. We represent the class of all multiple $(q,p)$-summing
	operators by $\Pi_{(q,p)}^{m}\left(  E_{1},\dots,E_{m};F\right)  $. When
	$E_{1}=\cdots=E_{m}=E$, we denote simply by $\Pi_{(q,p)}^{m}\left(
	^{m}E;F\right)  $.
\end{definition}

Following \cite{arr}, when we write $(x_{j})_{j=1}^{\infty}\in\ell_{r}\ell
_{s}^{\omega}(E)$, it means that there are $\left(  a_{j}\right)
_{j=1}^{\infty}\in\ell_{r}$ and $\left(  z_{j}\right)  _{j=1}^{\infty}\in
\ell_{s}^{w}(E)$ such that $x_{j}=a_{j}z_{j}$ for all $j$.

The following result of Arregui and Blasco \cite[Lemma 3 and Proposition 6]{arr} is crucial for us:

\begin{proposition}\label{Lemma3}
Let $1<r<\infty$. Then $\ell^{\omega}_{1}(E)=\ell_{r}\ell^{\omega}_{r^{\ast}}(E)$ if and only if $\mathcal{L}\left(  c_{0};E\right)  =\Pi_{r}\left(  c_{0};E\right)  $. In particular,
\begin{itemize}
\item[(a)] If $E$ has cotype $2,$ then $\ell_{1}^{\omega}(E)=\ell_{2}\ell_{2}^{\omega}(E)$.
\item[(b)] If $E$ has cotype $q>2$, then $\ell_{1}^{\omega}(E)=\ell_{r}\ell_{r^{\ast}}^{\omega}(E)$ for any $r>q$.
\end{itemize}
\end{proposition}

Note that from the previous result it is immediate that if $E$ has cotype $2$
then $\ell_{1}^{\omega}(E)=\ell_{r} \ell^{\omega}_{r^{\ast}}(E)$ for any $r\geq2$.

Let us to present now an inclusion theorem for multiple summing multilinear operators that, when restricted to the linear case, will be important in the proof of our main result. Before that, let us to recall the classical inclusion theorem for absolutely summing operators \cite[Theorem 10.4]{diestel}, which will be very useful to our propose.

\begin{theorem}\label{classicalinclusion}
If $1\leq a_{j}\leq b_{j}$ for $j=1,2$ and $\frac{1}{b_{1}}-\frac{1}{a_{1}}\leq\frac{1}{b_{2}}-\frac{1}{a_{2}}$, then every absolutely $\left(a_{1},b_{1}\right)$-summing operator is also absolutely $\left(a_{2},b_{2}\right)$-summing.
\end{theorem}

\begin{lemma}\label{inn}
Let $E_{1},...,E_{m}$ and $F$ be infinite dimensional Banach spaces.
\begin{itemize}
\item[(a)] If $E_{1},...,E_{m}$ has cotype $2$ and $r\leq2$, then
\[
\Pi_{\left(  a,r\right)  }\left(  E_{1},..., E_{m};F\right)  \subset\Pi_{\left(s,1\right)  }\left(   E_{1},...,E_{m};F\right)  \ \text{with} \ \frac{1}{r}-\frac{1}{a}\leq1-\frac{1}{s}.
\]
\item[(b)] If $E_{i}$ has cotype $q_{i}>2$, $q=\max\{q_{i}:i=1,...,m\}$ and $1<r<q^{\ast}$, then
\[
\Pi_{\left(  a,r\right)  }\left(  E_{1},..., E_{m};F\right)  \subset\Pi_{\left(s,1\right)  }\left(   E_{1},..., E_{m};F\right) \ \text{with} \ \frac{1}{r}-\frac{1}{a}\leq1-\frac{1}{s}.
\]
\end{itemize}
\end{lemma}

\begin{proof}
(a) If $r\leq 2$, then $r^*\geq 2$ and, from Theorem \ref{classicalinclusion}, $\Pi_2(E;F)\subseteq\Pi_{r^*}(E;F)$ for all Banach spaces $E$ and $F$. Thus, since $E_i$ has cotype $2$, it follows that $\mathcal{L}(c_0;E_i)=\Pi_2(c_0;E_i)\subseteq \Pi_{r^*}(c_0;E_i)$, that is, $\mathcal{L}(c_0;E_i)=\Pi_{r^*}(c_0;E_i)$. From Proposition \ref{Lemma3} we thus have $\ell_1^w(E_i)=\ell_{r^*}\ell_r^w(E_i)$. Then, if $(x_{j}^{(i)})_{j=1}^\infty\in\ell_{1}^{w}(E_{i})$,
\[
x_{j}^{(i)}=a_{j}^{(i)}z_{j}^{(i)}
\]
with $(  a_{j}^{(i)})  _{j}\in\ell_{r^{\ast}}$ and $(  z_{j}^{(i)})_{j}\in\ell_{r}^{w}(E_{i})$. Since
\[
\frac{1}{s}\leq\frac{1}{r^{\ast}}+\frac{1}{a},
\]
from H\"{o}lder's inequality we have
\begin{align*}
&  \left(  \sum\limits_{j_{1},...,j_{m}=1}^{\infty}\left\Vert T\left(x_{j_{1}}^{(1)},...,x_{j_{m}}^{(m)}\right)  \right\Vert ^{s}\right)^{\frac{1}{s}}\\
&  =\left(  \sum\limits_{j_{1},...,j_{m}=1}^{\infty}\left\vert a_{j_{1}}^{(1)}...a_{j_{m}}^{(m)}\right\vert ^{s}\left\Vert T\left(  z_{j_{1}}^{(1)},...,z_{j_{m}}^{(m)}\right)  \right\Vert ^{s}\right)  ^{\frac{1}{s}}\\
&  \leq\left(  \sum\limits_{j_{1},...,j_{m}=1}^{\infty}\left\vert a_{j_{1}}^{(1)}...a_{j_{m}}^{(m)}\right\vert ^{r^{\ast}}\right)  ^{\frac{1}{r^{\ast}}}\left(  \sum\limits_{j_{1},...,j_{m}=1}^{\infty}\left\Vert T\left(  z_{j_{1}}^{(1)},...,z_{j_{m}}^{(m)}\right)  \right\Vert ^{a}\right)  ^{\frac{1}{a}}\\
&  <\infty.
\end{align*}

(b) Follows as in (a) using Proposition \ref{Lemma3}(b).
\end{proof}

\begin{remark}
This theorem is a generalization of \cite[Theorem 10]{popa}.
\end{remark}

\begin{theorem}\label{444}
Let $a,b\in\lbrack1,\infty)$ and $E$ be an infinite dimensional Banach space.

\begin{itemize}
\item[(i)] If $b\geq\left(  \cot E\right)  ^{\ast}$, then
\[
\inf\left\{  a \ : \ id_{E}\text{ is absolutely }\left(  a,b\right)
\text{-summing}\right\}  =\infty.
\]
\item[(ii)] (ii) If $b<\left(  \cot E\right)^{\ast}$, then
\[
\inf\left\{  a \ : \ id_{E} \ \text{is absolutely} \ \left(  a,b\right)\text{-summing}\right\}  =\frac{b\cot E}{b+\cot E-b\cot E}.
\]
\end{itemize}
\end{theorem}

\begin{proof}
(i) Let us to prove that $id_{E}$ fails to be $\left(  a,\left(  \cot
E\right)  ^{\ast}\right)  $-summing; this result seems to have been overlooked
in the literature; it appears, in a more general form, in the preprint
\cite{bay} and we sketch the proof below. A famous result of Maurey and Pisier
asserts that $\ell_{\cot E}$ is finitely representable in $E,$ i.e., for all
$n\geq1$, there exists $E_{n}\subset E$ and an isomorphism $S_{n}:\ell_{\cot
E}^{n}\rightarrow E_{n}$ such that $\Vert S_{n}\Vert\cdot\Vert S_{n}^{-1}%
\Vert\leq2$. There is no loss of generality in assuming $\Vert S_{n}\Vert
\leq1$. For $1\leq i\leq n$, let
\[
y_{i}=S_{n}(e_{i}).
\]
Note that%
\[
1=\left\Vert e_{i}\right\Vert _{\ell_{\cot E}^{n}}=\left\Vert S_{n}^{-1}%
(y_{i})\right\Vert _{\ell_{\cot E}^{n}}\leq\left\Vert S_{n}^{-1}\right\Vert
\left\Vert y_{i}\right\Vert \leq2\left\Vert y_{i}\right\Vert ,
\]
and thus $\left\Vert y_{i}\right\Vert \geq1/2.$ Moreover
\begin{align*}
\sup_{\varphi\in B_{E^{\ast}}}\left(  \sum_{i=1}^{n}|\langle\varphi
,y_{i}\rangle|^{\left(  \cot E\right)  ^{\ast}}\right)  ^{\frac{1}{\left(
\cot E\right)  ^{\ast}}} &  =\sup_{\varphi\in B_{E^{\ast}}}\sup_{\alpha\in
B_{\ell_{\cot E}^{n}}}\sum_{i=1}^{n}\alpha_{i}\langle\varphi,y_{i}\rangle   \\
& =\sup_{\alpha\in B_{\ell_{\cot E}^{n}}}\sup_{\varphi\in B_{E^{\ast}}}\langle\varphi,\sum_{i=1}^{n}\alpha_{i}y_{i}\rangle\\
&  =\sup_{\alpha\in B_{\ell_{\cot E}^{n}}}\left\Vert \sum_{i=1}^{n}\alpha_{i}y_{i}\right\Vert _{E} \\
& =\sup_{\alpha\in B_{\ell_{\cot E}^{n}}}\left\Vert \sum_{i=1}^{n}\alpha_{i}S_{n}(e_{i})\right\Vert _{E}\\
&  =\sup_{\alpha\in B_{\ell_{\cot E}^{n}}}\left\Vert S_{n}\left(\alpha\right)  \right\Vert _{E} \\
& \leq\left\Vert S_{n}\right\Vert \sup_{\alpha\in B_{\ell_{\cot E}^{n}}}\left\Vert\alpha\right\Vert _{\ell_{\cot E}^{n}} \leq 1.
\end{align*}
We thus conclude that for any $a\geq1$ we have%
\[
\frac{n}{2^{a}}\leq\sum\limits_{j=1}^{n}\left\Vert id_{E}(y_{j})\right\Vert
^{a}\text{ and }\sup_{\varphi\in B_{E^{\ast}}}\sum_{i=1}^{n}|\langle
\varphi,y_{i}\rangle|^{\left(  \cot E\right)  ^{\ast}}\leq1
\]
and this means that $id_{E}$ is not absolutely $\left(  a,\left(  \cot
E\right)  ^{\ast}\right)  $-summing. The proof of (i) is done.



Now let us prove (ii). Consider $$\lambda=\frac{b\cot E}{b+\cot E-b\cot E}.$$ If $r<\lambda$, then
\[
\frac{1}{r}>\frac{1}{b}-1+\frac{1}{\cot E}.
\]
So
\[
\cot E- \dfrac{rb}{rb+b-r}>0.
\]
Thus, there is $0<\varepsilon<\cot E- \frac{rb}{rb+b-r}$ such that
\[
\frac{1}{b}-\frac{1}{r}<1-\frac{1}{\cot E-\varepsilon  }.%
\]
From Lemma \ref{inn}, with $m=1$ and $s=\cot E-\varepsilon$, we conclude that
\[
\Pi_{\left(  r,b\right)  }\left(  E;E\right)  \subset\Pi_{\left(
	\cot E-\varepsilon,1\right)  }\left(  E;E\right)  \neq\mathcal{L}\left(
E;E\right)
\]
and thus $id_{E}$ fails to be $\left(  r,b\right)  $-summing. So
\[
\lambda\leq\inf\left\{  a \ : \ id_{E} \ \text{is absolutely } \ \left(  a,b\right)\text{-summing}\right\} .
\]

Now, if $r>\lambda$ then
\[
\frac{1}{r}<\frac{1}{b}-1+ \frac{1}{\cot	E}.
\]
Therefore, there exist $\varepsilon>0$ such that

\[\frac{1}{r}\leq\frac{1}{b}-1+ \frac{1}{\cot	E+\varepsilon}.\]

Thus $r>\cot	E+\varepsilon$ and
\[
1 - \frac{1}{\cot	E+\varepsilon}\leq\frac{1}{b}-\frac{1}{r}.
\]
The inclusion theorem (Theorem \ref{classicalinclusion}) asserts that every absolutely $\left(  \cot	E+\varepsilon,1\right)$-summing operator is also absolutely $\left(r,b\right)$-summing. Since $id_{E}$ is absolutely $\left(  \cot	E+\varepsilon,1\right)$-summing, it follows that $id_{E}$ is also absolutely $\left( r,b\right)$-summing. Consequently,
\[
\lambda=\inf\left\{  a \ : \ id_{E} \ \text{is absolutely} \ \left(  a,b\right)\text{-summing}\right\},
\]
and the proof is done.
\end{proof}

\begin{figure}[h]
\begin{tikzpicture}
\path[draw,shade,top color=black,bottom color=black,opacity=.5]
(0,13/6) --
(1,11/5) --
(2,9/4) --
(3,7/3) --
(3.5,2. 4) --
(4,5/2) --
(9/2,8/3) --
(5,3) --
(51/10,28/9) --
(52/10,13/4) --
(53/10,24/7) --
(54/10,11/3) --
(55/10,4) --
(56/10,9/2) --
(5.7,16/3) --
(0,16/3) -- cycle;

\path[draw,shade,top color=black,bottom color=black,opacity=.2]
(0,0) --
(0,13/6) --
(1,11/5) --
(2,9/4) --
(3,7/3) --
(3.5,2. 4) --
(4,5/2) --
(9/2,8/3) --
(5,3) --
(51/10,28/9) --
(52/10,13/4) --
(53/10,24/7) --
(54/10,11/3) --
(55/10,4) --
(56/10,9/2) --
(5.7,16/3) --
(8,16/3) --
(8,13/6-0.5) -- cycle;

\draw[dotted] (6,0) -- (6,2.2);
\draw[dotted] (6,2.9) -- (6,16/3);
\draw[dotted] (0,0) -- (8,13/6-0.5);
\draw[dotted] (0,13/6-0.5) -- (8,13/6-0.5);
\draw[dotted] (8,0) -- (8,13/6-0.5);

\draw (5,2.5) node[right] {$\frac{1}{b}-\frac{1}{a}=\frac{1}{(\cot E)^*}$};

\draw[->] (0,0) -- (8.5,0) node[below] {$b$};
\draw (0,0) node[below] {$1$};
\draw (6,0) node[below] {$(\cot E)^*$};
\draw (8,0) node[below] {$2$};

\draw[->] (0,0) -- (0,35/6) node[left] {$a$};
\draw (0,0) node[left] {$1$};
\draw (0,13/6-0.5) node[left] {$2$};
\draw (0,13/6) node[left] {$\cot E$};

\path[draw,shade,top color=black,bottom color=black,opacity=.2]
(2,-1) -- (2,-1.5) -- (2.5,-1.5) -- (2.5,-1) -- cycle;
\draw (2.5,-1.3) node[right] {$id_E$ is not $(a,b)$-summing};

\path[draw,shade,top color=black,bottom color=black,opacity=.5]
(2,-1.75) -- (2,-2.25) -- (2.5,-2.25) -- (2.5,-1.75) -- cycle;
\draw (2.5,-2.05) node[right] {$id_E$ is $(a,b)$-summing};

\begin{scope}[shift={(0,0)}]
\draw[domain=0:5.7, ultra thick, smooth]
plot (\x,{-1/(\x-6)+2});
\end{scope}1
\end{tikzpicture}
\caption{Graphical overview of Theorem \ref{444}.}
\label{ddssdds}
\end{figure}
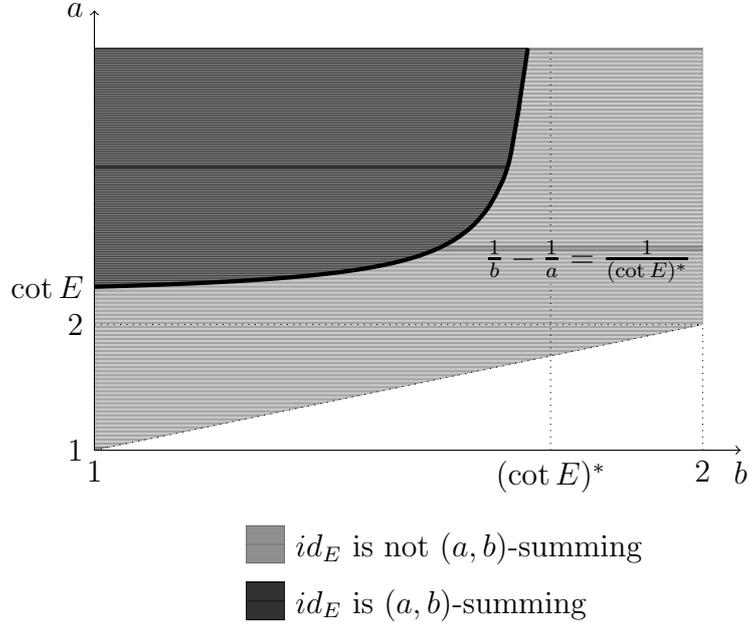

\begin{remark} Note that this result completes the information of the Dvoretzky--Rogers Theorem (Theorem \ref{dvo}) and, when $b=1$, we recover the classical result of Maurey and Pisier (Theorem \ref{pisier}).
\end{remark}


In the previous theorem we prove that for $b<\left(  \cot E\right)  ^{\ast}$ we can fully characterize the classes of $\left(  a,b\right)  $-summing operators that coincide. The Figure 1 illustrates this result and motivates us to name the next result as \textquotedblleft strings of coincidences\textquotedblright, which generalizes Theorem \ref{djt}.

\begin{theorem}[Strings of coincidences]\label{inn2}
Let $E$ and $F$ be infinite dimensional Banach spaces.

\begin{itemize}
\item[(a)] If $E$ has cotype $2$ and $a\leq a_{1}$ and $r\leq r_{1}\leq2$, then
\[
\Pi_{\left(  a,r\right)  }\left(  E;F\right)  =\Pi_{\left(  a_{1},r_{1}\right)  }\left(  E;F\right)  \ \text{with}\ \ \frac{1}{r}-\frac{1}%
{a}=\frac{1}{r_{1}}-\frac{1}{a_{1}}.
\]
\item[(b)] If $a\leq a_{1}$ and $1\leq r\leq r_{1}<(\cot E)^{\ast}$, then
\[
\Pi_{\left(  a,r\right)  }\left(  E;F\right)  =\Pi_{\left(  a_{1},r_{1}\right)  }\left(  E;F\right)  \ \text{with}\ \ \frac{1}{r}-\frac{1}%
{a}=\frac{1}{r_{1}}-\frac{1}{a_{1}}.
\]
\end{itemize}
\end{theorem}

\begin{proof}
	Note that	
	\begin{equation}\label{coin}
	\Pi_{\left(  a,r\right)  }\left(  E;F\right)  =\Pi_{\left(s,1\right)  }\left(  E;F\right)=\Pi_{\left(a_1,r_1\right)  }\left(  E;F\right),
	\end{equation}
	with $$1-\frac{1}{s}=\frac{1}{r}-\frac{1}{a}=\frac{1}{r_{1}}-\frac{1}{a_1}.$$  In fact, the two equalities in (\ref{coin}) are immediate consequences of the previous theorem and the inclusion theorem for absolutely summing operators (Theorem \ref{classicalinclusion}).
\end{proof}

\begin{figure}[h]
\begin{tikzpicture}
\draw[dotted] (10,1) -- (20,2);
\draw[dotted] (10,2) -- (20,2);
\draw[dotted] (20,1) -- (20,2);

\draw[->] (10,1) -- (20.5,1) node[below] {$r$};
\draw (10,1) node[below] {$1$};
\draw (20,1) node[below] {$2$};

\draw[->] (10,1) -- (10,6.5) node[left] {$a$};
\draw (10,1) node[left] {$1$};
\draw (10,2) node[left] {$2$};

\begin{scope}[shift={(0,0)}]
\draw[opacity=1, domain=10:15, ultra thick, smooth]
plot (\x,{2*\x/(20-\x)});
\end{scope}

\begin{scope}[shift={(0,0)}]
\draw[opacity=.6, domain=10:1500/85, ultra thick, smooth]
plot (\x,{2.5*\x/(25-\x)});
\end{scope}

\begin{scope}[shift={(0,0)}]
\draw[opacity=.5, domain=10:20, ultra thick, smooth]
plot (\x,{3*\x/(30-\x)});
\end{scope}

\begin{scope}[shift={(0,0)}]
\draw[opacity=.4, domain=10:20, ultra thick, smooth]
plot (\x,{5*\x/(50-\x)});
\end{scope}

\begin{scope}[shift={(0,0)}]
\draw[opacity=.3, domain=10:20, ultra thick, smooth]
plot (\x,{7*\x/(70-\x)});
\end{scope}

\begin{scope}[shift={(0,0)}]
\draw[opacity=.1, domain=10:20, ultra thick, smooth]
plot (\x,{10*\x/(100-\x)});
\end{scope}
\end{tikzpicture}
\caption{Strings of coincidence for $(a,r)$-summability when $\cot E=2$.}
\label{aaaaaaaaaaaaa}
\end{figure}
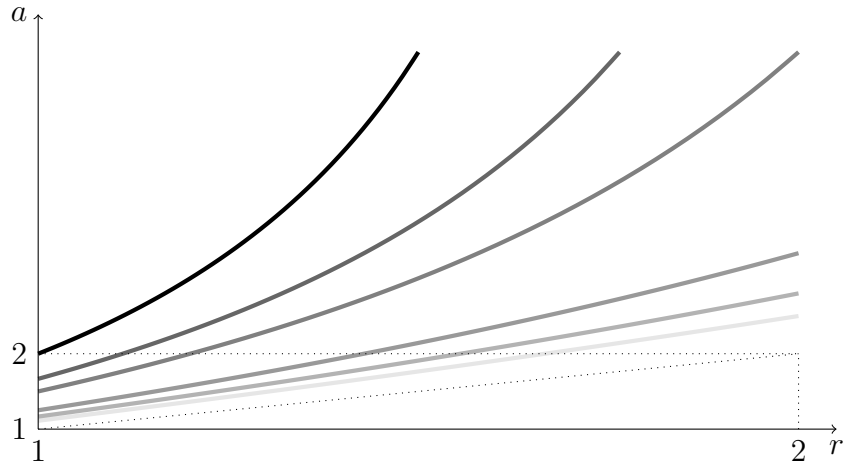

\begin{remark}
The strings of coincidence can also be obtained as a consequence of \cite[Theorem 2.1]{thiago}.
\end{remark}


We finish this section by illustrating how the strings of coincidence can be useful to extend the classical result of Kwapie\'{n} (Theorem \ref{kw}).

\begin{theorem}
	Let $1\leq p\leq\infty$ and $a,b\geq1.$
	If $b\leq2$, then $	\Pi_{\left(  a,b\right)  }\left(  \ell_1;\ell_p\right)  =\mathcal{L}\left(  \ell_1;\ell_p\right)$
	if, and only if, $a\geq\frac{br(p)}{r(p)+b-br(p)}$ where $\frac{1}{r(p)}=1- \left|\frac{1}{p}-\frac{1}{2}\right|$.
\end{theorem}

\begin{proof}
For $b\leq2$ the string of coincidence associated to $\left(  r(p),1\right)$, where $\frac{1}{r(p)}=1- \left|\frac{1}{p}-\frac{1}{2}\right|$, is composed by the pairs $\left(  a,b\right)  $
	such that%
	\[
	\frac{1}{b}-\frac{1}{a}=1-\frac{1}{r(p)}.
	\]
	Thus%
	\[
	a=\frac{br(p)}{r(p)+b-br(p)}.
	\]
	In view of the optimality of $\frac{br(p)}{r(p)+b-br(p)}$ when $b=1$, we can conclude that the above estimate of $a$ is sharp.
\end{proof}

\end{document}